\documentclass[12pt]{amsart}

\usepackage{amsfonts}
\usepackage{amsmath,amscd}
\usepackage{amsthm}
\usepackage{url}
\usepackage[all]{xy}
\usepackage{graphicx}
\usepackage{latexsym}
\usepackage{amssymb}
\usepackage[cp850]{inputenc}
\usepackage{epsfig}
\usepackage{psfrag}
\usepackage{amsfonts}
\usepackage{dsfont}

\newtheorem{sat}{Theorem}[section]		\newtheorem{lem}[sat]{Lemma}
\newtheorem{kor}[sat]{Corollary}			\newtheorem{prop}[sat]{Proposition}
				\newtheorem{defi}{Definition}
\newtheorem*{defi*}{Definition}			\newtheorem*{bei*}{Example}
\newtheorem*{sat*}{Theorem}				\newtheorem*{kor*}{Corollary}
\newtheorem*{rmk*}{Remark}					

\let\ssection=\section
\renewcommand{\section}{\setcounter{equation}{0}\ssection}

\newtheorem*{namedtheorem}{\theoremname}
\newcommand{\theoremname}{testing}

\theoremstyle{remark}

\newcommand{\BC}{\mathbb C}

\newcommand{\BS}{\mathbb S}			\newcommand{\BZ}{\mathbb Z}

\newcommand{\BONE}{\mathds 1}

		\newcommand{\CH}{\mathcal H}

\newcommand{\actson}{\curvearrowright}

\DeclareMathOperator{\SL}{SL}		
\DeclareMathOperator{\Hom}{Hom}		

\DeclareMathOperator{\SU}{SU}

\DeclareMathOperator{\Ker}{Ker}

\newcommand{\comment}[1]{}

\DeclareMathOperator{\SO}{SO}

\DeclareMathOperator{\Spin}{Spin}
\DeclareMathOperator{\U}{U}

\DeclareMathOperator{\Sp}{Sp}

\newcommand{\x}{{\bf x}}
\newcommand{\y}{{\bf y}}

\begin{document}

\title[On the fundamental group of $\Hom(\BZ^k,G)$]{On the fundamental group of $\bf{\text{Hom}({\mathbb{\BZ}}^k,G)}$}
\author[J. Manuel G\'omez, A. Pettet, J. Souto]{Jos\'e Manuel G\'omez , Alexandra Pettet, and Juan Souto}
\thanks{The second author has been partially supported by NSF grant DMS-0856143 and NSF RTG grant DMS-0602191. The third author has been partially supported by the NSF grant DMS-0706878, NSF Career award 0952106 and the Alfred P. Sloan Foundation}

\begin{abstract}
Let $G$ be a compact Lie group. Consider the variety $\Hom(\BZ^{k},G)$ of representations of $\BZ^k$ into $G$. We can see this as a based space 
by taking as base point the trivial representation $\BONE$. The goal of this paper is to prove that $\pi_1(\Hom(\BZ^k,G))$ is naturally isomorphic to $\pi_1(G)^k$.
\end{abstract}
\maketitle

\section{Introduction}
Let $G$ be a compact Lie group. The set $\Hom(\BZ^k,G)$ can naturally be identified with the subset of $G^{k}$ consisting 
of ordered commuting $k$-tuples in $G$. In this way, $\Hom(\BZ^k,G)$ can be given a topology as a subspace of $G^{k}$ making it into a, possibly singular, real analytic variety. Let $\BONE\in \Hom(\BZ^k,G)$ be the trivial representation. Then $\Hom(\BZ^k,G)$ can be seen as a based space with base point $\BONE$. As announced in the abstract, the goal of this paper is to prove the following result:

\begin{sat}\label{main}
Let $G$ be a compact Lie group. Then for every $k\ge 1$ there is a natural isomorphism 
$$\pi_1(\Hom(\BZ^k,G))\cong\pi_1(G)^k.$$
\end{sat}

Theorem \ref{main} is due to Torres-Giese and Sjerve \cite{Enrique} in the case that $G$ is either $\SO(3),\SU(2)$ or $\U(2)$. In their work, Torres-Giese and Sjerve determine the topological type of $\Hom(\BZ^k,G)$ and compute its fundamental group via the Seifert-van Kampen Theorem. 
Our approach is as follows. Let $G$ be a compact Lie group and denote by $G_{0}$ the connected component of $G$ containing the unit $1_{G}$. The natural inclusion $i:G_{0}\hookrightarrow G$ gives rise to a map 
\[
\Hom(\BZ^{k},G_{0})\stackrel{i_{*}}{\rightarrow}\Hom(\BZ^{k},G)
\]
that induces an isomorphism of fundamental groups. Therefore, we can assume  without loss of generality that $G$ is a compact connected Lie group. 
Observe that in general $\Hom(\BZ^k,G)$ is not connected even if $G$ is connected and simply connected. We denote by $\Hom(\BZ^k,G)_\BONE$ the connected component of $\Hom(\BZ^k,G)$ containing the trivial representation $\BONE$. We are thus interested in computing $\pi_{1}(\Hom(\BZ^k,G)_\BONE)$.
Fix $T$ a maximal torus in $G$, let $N(T)$ be the normalizer of $T$ in $G$ and $W=N(T)/T$ the associated Weyl group. Following Baird \cite{Baird}, we consider the continuous surjection 
\begin{align*}
\sigma_{k}:G/T\times_{W}T^{k}=G\times_{N(T)}T^k&\to\Hom(\BZ^k,G)_\BONE\\
[(g,t_{1},\dots,t_{k})]&\mapsto (gt_{1}g^{-1},\dots,gt_{k}g^{-1}).
\end{align*}
When $k=1$ this map corresponds to the classical map given by conjugation
\begin{align*}
\sigma_{1}:G/T\times_{W} T&\to G\\
([g],t)&\mapsto gtg^{-1}.
\end{align*}
If $G^{reg}$ denotes the subspace of regular elements in $G$, then Weyl's covering theorem (see \cite[Theorem 3.7.2]{Duistermaat}) asserts that the restriction of $\sigma_{1}$ to $G/T\times_{W}(G^{reg}\cap T)$ 
is a $G$-equivariant real-analytic diffeomorphism onto $G^{reg}$. An analogous result is true in general for $k\ge 2$. The map $\sigma_{k}$ is the main tool that we will use to compute $\pi_{1}(\Hom(\BZ^k,G)_\BONE)$. Using a general position argument we show that the map $\sigma_{k}$ is $\pi_1$-surjective. Then, under the additional assumption that $G$ is simply connected, we show that every element in a suitable generating set of $\pi_1(G\times_{N(T)} T^k)$ is in the kernel of $\pi_1(\sigma_{k})$. At this point we will have proved Theorem \ref{main} in the case that $\pi_1(G)$ is trivial. To finish, we reduce the general case to the simply connected case by passing to a suitable cover $\tilde G$ of the group $G$ and studying the relation between $\Hom(\BZ^k,G)_\BONE$ and $\Hom(\BZ^k,\tilde G)_\BONE$.
\medskip

In the course of the proof of Theorem \ref{main} we will need in a key way that $G$ is compact because otherwise the map $\sigma_k$ above will fail to be surjective. However, we would like to mention that the two last authors of this note have proved in \cite{retract} that the inclusion of $\Hom(\BZ^k,\SU(n))$ into $\Hom(\BZ^k,\SL(n,\BC))$ is a homotopy equivalence. Since $\Hom(\BZ^k,\SU(n))$ is connected \cite{KS} and $\SU(n)$ is simply connected for $n\ge 2$, we deduce from Theorem \ref{main}:

\begin{kor}
$\Hom(\BZ^k,\SL(n,\BC))$ is connected and simply connected for all $k\ge 1$ and all $n\ge 2$.\qed
\end{kor}

This paper is organized as follows. In section \ref{sec-simply} we prove Theorem \ref{main} for simply connected groups. In section \ref{sec-not} we extend this to general compact Lie groups. Finally in section \ref{examples} we discuss some examples showing that Theorem \ref{main} fails if the base point of $\Hom(\BZ^{k},G)$ is no longer assumed to be in $\Hom(\BZ^{k},G)_{\BONE}$.

\medskip

\noindent{\bf Acknowledgments:} We would like to thank Mladen Bestvina, Alejandro \'Adem and Enrique Torres-Giese for very pleasant and informative conversations.
\medskip

\noindent{\bf{Remark:}}
The space $\Hom(\BZ,G)$ is naturally homeomorphic to $G$; hence, Theorem \ref{main} is trivially satisfied for $k=1$. Therefore, we will assume from now on that $k\ge 2$. Also, note that Theorem \ref{main} holds trivially for finite groups. Thus we can also assume that $G$ has rank at least $1$.

\section{The simply connected case}\label{sec-simply}
In this section we prove Theorem \ref{main} for the particular case where $G$ is a simply connected Lie group. 

\medskip

From now on fix a compact connected Lie group $G$ and $T$ a maximal torus in $G$. Let $N(T)$ be the normalizer of $T$ in $G$ and denote by 
$W=N(T)/T$ the Weyl group associated to $T$. To $g\in G$ and $t_1,\dots,t_k\in T$ we can associate the representation 
\begin{align*}
\rho_{(g,t_1,\dots,t_k)}:\BZ^k&\to G\\
(n_1,\dots,n_k)&\mapsto gt_1^{n_1}\dots t_k^{n_k}g^{-1}.
\end{align*}
This way we obtain a continuous map
\begin{align*}
\tilde{\sigma}_{k}:G\times T^k&\to\Hom(\BZ^k,G)\\
(g,t_{1},\dots,t_{k})&\mapsto \rho_{(g,t_1,\dots,t_k)}
\end{align*}
which is constant along the orbits of the diagonal action of $N(T)$ onto $G\times T^k$. Thus we have an induced map 
$$\sigma_{k}:G\times_{N(T)} T^k\to\Hom(\BZ^k,G).$$
Observe that 
\[
G\times_{N(T)} T^k=G/T\times_{W}T^{k}
\] 
is a real-analytic manifold and that the map $\sigma_{k}$ is a morphism of real-analytic spaces. Moreover, since $W$ acts freely on 
$G/T$ then the projection onto the first factor induces a fibration sequence of the form 
\[
T^{k}\to G/T\times_{W}T^{k}\to G/N(T).
\]
Let $\Hom(\BZ^k,G)_\BONE$ be the connected component of $\Hom(\BZ^k,G)$ containing the trivial representation
\begin{align*}
\BONE:\BZ^k&\to G\\ 
(n_1,\dots,n_k)&\mapsto 1_G.
\end{align*}
In \cite{Baird}, Baird studied properties of the map $\sigma_{k}$. For instance, by \cite[Lemma 4.2]{Baird} the map $\sigma_{k}$ is a surjection onto $\Hom(\BZ^k,G)_\BONE$, and this space is precisely the subspace of $\Hom(\BZ^k,G)$ consisting of commuting $k$-tuples contained in some maximal torus of $G$. Also, by \cite[Theorem 4.3]{Baird}, the fibers of $\sigma_{k}$ have the cohomology of a point if one has for example coefficients over a field of characteristic $0$. From this we deduce in particular that the fibers of $\sigma_{k}$ are connected. These facts are summarized in the following proposition

\begin{prop}\label{blabla}
The space $\Hom(\BZ^{k},G)_{\BONE}$ is precisely the subspace of $\Hom(\BZ^k,G)$ of commuting $k$-tuples contained in some maximal torus of $G$,
and the map $\sigma_{k}:G\times_{N(T)} T^k\to\Hom(\BZ^k,G)_\BONE$ is surjective and has connected fibers.
\end{prop}

The map $\sigma_{k}$ is certainly not injective; however, there is a large set on which it has this  desirable property. Recall that the action of $N(T)$ on $T$ by conjugation induces an action $W\actson T$. We denote by $(T^k)^*$ the subset of $T^k$ consisting of all $k$-tuples $(t_1,\dots,t_k)$ with the property that the trivial element is the only element in $W$ which fixes $t_i$ for $i=1,\dots,k$. Clearly, $T^k\setminus(T^k)^{*}$ is a compact analytic subset of co-dimension at least $k$ because $(T^k)^*$ contains the subspace of $k$-tuples $(t_{1},\dots,t_{k})\in T^{k}$ for which at least one of the $t_{i}$'s is regular. Therefore we obtain the following:

\begin{lem}\label{co-dimension}
The complement of $G\times_{N(T)}(T^k)^*$ in $G\times_{N(T)} T^k$ is a compact analytic subset of co-dimension at least $k\ge 2$.\qed
\end{lem}

The open set $G\times_{N(T)}(T^k)^*$ will be important to us because it is homeomorphic to a very large subset of $\Hom(\BZ^k,G)_\BONE$:

\begin{lem}\label{homeo}
The restriction of the map $\sigma_{k}$ to $G\times_{N(T)}(T^k)^*$ is a homeomorphism onto its image.
\end{lem}

\begin{proof}
Note that $G\times_{N(T)}T^{k}$ is a compact space and $\sigma_{k}:G\times_{N(T)} T^k\to\Hom(\BZ^k,G)_\BONE$ is a continuous map. In particular, 
$\sigma_{k}$ is a closed map. This shows that the restriction of $\sigma_{k}$ to the open subspace $G\times_{N(T)}(T^k)^*$ is a continuous, closed and 
surjective map onto its image. Therefore it suffices to see that restriction of $\sigma_{k}$ to $G\times_{N(T)}(T^k)^*$ is injective. Suppose that $(g,t_1,\dots,t_k)$ and $(h,s_1,\dots,s_k)$ in $G\times(T^k)^*$ are such that
$$\rho_{(g,t_1,\dots,t_k)}=\rho_{(h,s_1,\dots,s_k)}.$$
This implies that $gt_ig^{-1}=hs_ih^{-1}$ and hence, $(h^{-1}g)t_i(h^{-1}g)^{-1}=s_{i}\in T$ for $i=1,\dots,k$. The assumption that $(t_1,\dots,t_k)\in(T^k)^*$ implies that $h^{-1}g$ normalizes $T$. In particular $(g,t_1,\dots,t_k)$ and $(h,s_1,\dots,s_k)$ represent the same point in $G\times_{N(T)} T^k$ as we needed to prove.
\end{proof}

\begin{defi} Define $\CH^r$ to be the image of $G\times_{N(T)}(T^k)^*$ under the map $\sigma_{k}$. We will refer to $\CH^r$ as the {\em regular part} of $\Hom(\BZ^k,G)_\BONE$. Also define $\CH^s=\Hom(\BZ^k,G)_\BONE\setminus\CH^r$, the complement of the regular part in $\Hom(\BZ^{k},G)_{\BONE}$. We will refer to $\CH^s$ as the {\em singular part} of $\Hom(\BZ^k,G)_\BONE$.
\end{defi}

\noindent{\bf{Remark:}} The subspace $\CH^r$ is precisely the set of all representations $\rho:\BZ^k\to G$ whose image has a maximal torus as its Zariski closure; we will not need this fact.

\begin{lem}\label{key}
The singular part $\CH^s$ is nowhere dense and does not disconnect connected open subsets of $\Hom(\BZ^k,G)_\BONE$.
\end{lem}

\begin{proof}
The fact that $\CH^s$ is nowhere dense follows from the fact that $\sigma_{k}$ is surjective and that the preimage $G\times_{N(T)}(T^k)^*$ of its complement $\CH^r=\Hom(\BZ^k,G)_\BONE\setminus\CH^s$ is dense in $G\times_{N(T)} T^k$.

We prove now that $\CH^s$ does not separate any connected open set $U\subset\Hom(\BZ^k,G)_\BONE$. Suppose that we have such a set $U$; if $U\cap\CH^s=\emptyset$ then there is nothing to prove, so also suppose that this is not the case. Then the connectivity of the fibers of $\sigma_{k}$ implies that the preimage $\sigma_{k}^{-1}(U)$ of $U$ under the surjective map $\sigma_{k}$ is connected as well. On the other hand, 
$$\sigma_{k}^{-1}(\CH^s)=(G\times T^k)/N(T)\setminus(G\times(T^k)^*)/N(T)$$ 
has co-dimension at least $k\ge 2$ by Lemma \ref{co-dimension}. 
A set of co-dimension at least 2 in a manifold does not disconnect connected open sets, and hence $\sigma_{k}^{-1}( U\setminus\CH^s)$ is connected. As $\sigma_{k}^{-1}(U\setminus\CH^s)$ and $U\setminus\CH^s$ are homeomorphic, by Lemma \ref{homeo}, we have that $\CH^s$ does not disconnect connected open sets.
\end{proof}

Recall that $\Hom(\BZ^k,G)_\BONE$ is real analytic and that, as the image of the compact analytic set 
$$(G\times_{N(T)} T^k)\setminus(G\times_{N(T)}(T^k)^*)$$
under the analytic map $\sigma_{k}$, the subset $\CH^s$ is closed and analytic. In particular, by the Whitney stratification theorem \cite{Thom,whitney}, $\Hom(\BZ^k,G)_\BONE$ admits the structure of a simplicial complex in such a way that $\CH^s$ is a subcomplex. The following lemma gives the reason we proved Lemma \ref{key} at all:

\begin{lem}\label{pushout}
Let $X$ be a compact simplicial complex and $Y\subset X$ a subcomplex. Suppose that $X\setminus Y$ is dense, and that $Y$ does not separate any connected open set in $X$. If  $x_{0}\in X\setminus Y$ is the basepoint, then $\pi_1(X\setminus Y,x_{0})$ surjects onto $\pi_1(X,x_{0})$. 
\end{lem}

\begin{proof}
Suppose that we have a loop $\gamma$ in $X$ based at $x_{0}$; we want to homotope $\gamma$ away from $Y$. Note that, as $Y$ is a nowhere dense subcomplex, we can homotope $\gamma$ so that it meets $Y$ in a finite number of points; suppose that $\gamma$ has been chosen to minimize this number. 

Seeking a contradiction, assume that $\gamma$ meets $Y$ at some point $p$, and let $U\subset X$ be a small open contractible neighborhood of $p$. Let also $J\subset U$ be the proper subarc of $\gamma$ containing $p$ and let $p_\pm\notin Y$ be the endpoints of $J$. Since $U\setminus Y$ is connected, we can connect $p_{\pm}$ inside $U\setminus Y$ by some arc $I$. Since $U$ is contractible, both $I$ and $J$ are homotopic to each other in $U$ while fixing $p_{\pm}$. It follows that we can replace the curve $\gamma$ by a homotopic curve which meets $Y$ in a point less than $\gamma$ did. This is not possible by the choice of $\gamma$, so the lemma follows.
\end{proof}

\noindent{\bf{Remark:}} Lemma \ref{pushout} can be proved in greater generality, but we will only need the version presented here.

\medskip

From Lemma \ref{key} and Lemma \ref{pushout}, it follows that $\pi_1(\CH^r)$ surjects onto $\pi_1(\Hom(\BZ^k,G)_{\BONE})$. On the other hand, $\CH^r$ is the homeomorphic image of $G\times_{N(T)}(T^k)^*$ under the map $\sigma_{k}$ by Lemma \ref{homeo}. Hence, we deduce the following:

\begin{kor}\label{pi1-surject}
If $G$ is a compact connected Lie group, the map $$\sigma_{k}:G\times_{N(T)} T^k\to\Hom(\BZ^k,G)_\BONE$$ is $\pi_1$-surjective.\qed
\end{kor}

Our next goal is to prove that the homomorphism
$$\pi_1(\sigma_{k}):\pi_1(G\times_{N(T)}\ T^k)\to\pi_1(\Hom(\BZ^k,G)_\BONE)$$
is trivial if we further assume that $G$ is simply connected. Recall that $\Hom(\BZ^k,G)_\BONE$ is a based space, with base point $\BONE$. We can also view $G\times_{N(T)}\ T^k$ as a based space by taking as base point the class representing the element $(1_G,\dots,1_G)\in G\times T^k$. With this choice of base points, the map $\sigma_{k}$ is a based map.

We will show that $\pi_1(\sigma_{k})$ is the trivial map by showing that a suitable set of generators of $\pi_1(G\times_{N(T)} T^k)$ is in the kernel of $\pi_{1}(\sigma_{k})$. In order to describe such a set of generators, recall that projection onto the first factor $p_{1}:G\times_{N(T)} T^k\to G/N(T)$ induces a fibration sequence of the form
\begin{equation}\label{fibration}
T^k\to G\times_{N(T)}T^k\stackrel{p_{1}}{\rightarrow} G/N(T).
\end{equation}
The tail end of the associated homotopy long exact sequence is the following exact sequence: 
\begin{equation}\label{homotopy-seq}
\pi_1(T^k)\to\pi_1(G\times_{N(T)}T^k)\stackrel{(p_{1})_{*}}{\rightarrow}\pi_1(G/N(T))\to 1.
\end{equation}
Observe that the map $p_{1}$ admits a section
\begin{align*}
s:G/N(T)&\to G\times_{N(T)}(T^k)\\ 
[g]&\mapsto[g,(1_G,\dots,1_G)]
\end{align*}
where $1_G$ is the unit element in $G$, and $[\cdot]$ denotes the class of the corresponding element in $G/N(T)$ and $G\times_{N(T)}T^k$, respectively. This section gives a splitting of the sequence \eqref{homotopy-seq}. We deduce:

\begin{lem}\label{generators}
$\pi_1(G\times_{N(T)}T^k)$ is generated by $\pi_1(\{1_G\}\times T^k)$ and by $\pi_1(s(G/N(T)))$.\qed
\end{lem}

At this point we would like to notice that the composition of the section $s$ with the map $\sigma_{k}$ is the constant map; the image is namely the trivial representation $\BONE$. It follows that
$$\pi_1(s(G/N(T)))\subset\Ker(\pi_1(\sigma_{k})).$$
In particular, by Lemma \ref{generators}, in order to show that $\pi_1(\sigma_{k})$ is trivial it suffices to show that the restriction of the map $\sigma_{k}$ to the fiber $\{1_G\}\times T^k$ is trivial in $\pi_1$.  We do this next. Identifying
$$\pi_1(T^k)=\pi_1(T)\times\dots\times\pi_1(T)$$
we see that $\pi_1(T^k)$ is generated by loops which are constant on each component but one. More concretely, for every $1\le a\le k$ let 
\begin{align*}
i_{a}:T&\to T\times\dots\times T\\
x&\mapsto (1_{G},\dots,x,\dots,1_{G}).
\end{align*}
be the natural inclusion of $T$ into the $a$-th factor of $T\times\cdots\times T$. Then $\pi_1(T^k)$ is generated by loops of the form $\eta(t)= i_{a}(\gamma(t))$, where $\gamma:[0,1]\to T$ is a loop in $T$ based at $1_G$.  Note that the image of a loop of the form $i_{a}(\gamma)$ under $\sigma_{k}$ is a loop $(\rho_t)$ in $\Hom(\BZ^k,G)_{\BONE}$ where each $\rho_t=\sigma_{k}(\eta(t))$ is given by
\[
\rho_t(n_1,\dots,n_k)=i_{a}(\gamma(t)^{n_a}),
\]
where here by abuse of notation we also denote by 
\begin{align*}
i_{a}:G&\to \Hom(\BZ^{k},G)\\
g&\mapsto (1_{G},\dots,g,\dots,1_{G})
\end{align*}
the inclusion of $G$ into the $a$-th factor of $\Hom(\BZ^{k},G)\subset G^{k}$. By assumption, and this is the first and only time that we use this assumption, $\pi_1(G)$ is trivial. Hence, the loop $\gamma(t)$ can be contracted in $G$ to the trivial loop. Let 
\begin{align*} 
[0,1]\times[0,1]&\to G\\ 
(s,t)&\mapsto\gamma^s(t)
\end{align*} 
be such a homotopy with $\gamma^0(t)=\gamma(t)$ and with $\gamma^1(t)=1_G$ for all $t$. Consider the homotopy
\begin{align*}
[0,1]\times[0,1]&\to\Hom(\BZ^k,G)_{\BONE},\\ 
(t,s)&\mapsto\rho_t^s
\end{align*}
where 
\[
\rho_t^s(n_1,\dots,n_k)=i_{a}(\gamma^s(t)^{n_a}).
\] 
This homotopy begins with the loop $(\rho_t)=\sigma_{k}(\eta)$ and ends with the constant curve with image the trivial representation $\BONE$. We have proved that the restriction of $\sigma_{k}$ to the fiber $\{1_G\}\times T^k$ of the fibration 
\[
G\times_{N(T)} T^k\to G/N(T)
\]
is trivial in $\pi_1$. Combining this fact with our earlier observations,  we deduce that $\pi_1(\sigma_{k})$ is the trivial homomorphism. On the other, by  Lemma \ref{pi1-surject}, the map $\pi_1(\sigma_{k})$ is surjective. This proves that $$\pi_1(\Hom(\BZ^k,G)_\BONE)=1.$$
In conclusion, we have proved the following theorem: 

\begin{sat}\label{casesimplyconnected}
Let $G$ be a simply connected compact Lie group. If $\Hom(\BZ^{k},G)$ has base point $\BONE$, then 
\[
\pi_{1}(\Hom(\BZ^{k},G))=1.
\]
\end{sat}
\noindent This is precisely Theorem \ref{main} in the case where $G$ is simply connected.

\section{The general case}\label{sec-not}

In this section we prove Theorem \ref{main} for any compact Lie group $G$. 

\medskip

To begin with, suppose that $G$ is a compact Lie group. Denote by $G_{0}$ the connected component of $G$ containing $1_{G}$. As mentioned in the introduction, the natural inclusion $i:G_{0}\hookrightarrow G$ gives rise to a map 
\[
\Hom(\BZ^{k},G_{0})\stackrel{i_{*}}{\rightarrow}\Hom(\BZ^{k},G)
\]
that induces an isomorphism of $\pi_{1}$ for any $k$. Because of this we only need to consider the case where $G$ is a compact connected Lie group.
Suppose then that $G$ is such a Lie group. By \cite[Theorem 6.19]{HM} we can write $G=\tilde{G}/K$, where $K$ is a finite subgroup in the center of $\tilde{G}$, and where
\[
\tilde{G} =(\BS^1)^r\times G_{1}\times\cdots\times G_{s}
\]
for some compact simply connected  and simple Lie groups $G_{1},\dots,G_{s}$. If we write 
\[
H=G_{1}\times\cdots\times G_{s}
\]
then $\tilde{G}=(\BS^1)^r\times H$ and $H$ is a compact and simply connected Lie group. Notice that the projection map 
\[
p:\tilde{G}\to G
\]
is both a homomorphism and a covering map, with covering group $K$; in particular, it is a local isomorphism.
In \cite[Lemma 2.2]{Goldman}, Goldman showed that if $\pi$ is a finitely generated group and $p:G'\to G$ is a local isomorphism, 
then  composition with $p$ defines a continuous map 
\[
p_{*}:\Hom(\pi,G')\to \Hom(\pi,G),
\]
such that the image of $p_{*}$ is a union of connected components of 
$\Hom(\pi,G)$. Moreover, if $Q$ is a connected component in the image of $p_{*}$, then the restriction of $p_{*}$
\[
(p_{*})_{|p_{*}^{-1}(Q)}:p_{*}^{-1}(Q)\to Q
\] 
is a covering space, with covering group $\Hom(\pi,K)$. We can apply this to the particular case of  $\pi=\BZ^{k}$ and $Q=\Hom(\BZ^{k},G)_{\BONE}$. Thus we obtain a covering space 
\[
p^{-1}(\Hom(\BZ^{k},G)_{\BONE})\to 
\Hom(\BZ^{k},G)_{\BONE} 
\]
with covering group $K^{k}=\Hom(\BZ^{k},K)$. For this covering space, the action of $K^{k}$ on $p^{-1}(\Hom(\BZ^{k},G)_{\BONE})$ corresponds 
to left component-wise multiplication. By Lemma \ref{blabla}, the space $\Hom(\BZ^{k},\tilde{G})_{\BONE}$ is precisely the subspace of $\Hom(\BZ^k,\tilde{G})$ of commuting $k$-tuples contained in some maximal torus of $\tilde{G}$. Using this and the fact that in any compact Lie group the center is contained in 
any maximal torus (see for example \cite[Corollary 4.47]{{Knapp}}), it follows that 
\[
p^{-1}(\Hom(\BZ^{k},G)_{\BONE})=\Hom(\BZ^{k},\tilde{G})_{\BONE}.
\] 
This shows that we have a covering space sequence 
\[
K^{k}\stackrel{i_{*}}{\rightarrow}\Hom(\BZ^{k},\tilde{G})_{\BONE}\stackrel{p_{*}}{\rightarrow} \Hom(\BZ^{k},G)_{\BONE}.
\]
The long exact sequence in homotopy  associated to this covering space shows that there is a short exact sequence
\begin{equation}\label{exacthomotopy}
1\to \pi_{1}(\Hom(\BZ^{k},\tilde{G})_{\BONE}) \stackrel{p_{*}}{\rightarrow} \pi_{1}(\Hom(\BZ^{k},G)_{\BONE})\stackrel{\delta}{\rightarrow} K^{k}\to 1.
\end{equation} 
On the other hand  there is a natural homeomorphism
\[
\Hom(\BZ^{k},\tilde{G})_{\BONE}\cong \Hom(\BZ^{k},(\BS^{1})^{r})\times \Hom(\BZ^{k},H)_{\BONE}.
\]
In particular 
\[
\pi_{1}(\Hom(\BZ^{k},\tilde{G})_{\BONE})\cong \pi_{1}(((\BS^1)^{r})^{k})\times \pi_{1}(\Hom(\BZ^{k},H)_{\BONE}).
\]
As $H$ is a compact and simply connected Lie group, by Theorem \ref{casesimplyconnected} we have $\pi_{1}(\Hom(\BZ^{k},H)_{\BONE})=1$. Thus 
\[
\pi_{1}(\Hom(\BZ^{k},\tilde{G})_{\BONE})\cong \pi_{1}(((\BS^1)^{r})^{k})\cong (\BZ^{r})^{k},
\]
with an isomorphism induced by the inclusion map 
\begin{align*}
(\BS^{1})^{r}&\to (\BS^{1})^{r}\times H=\tilde{G}\\
x&\mapsto (x,1).
\end{align*}

This shows that (\ref{exacthomotopy}) is a short exact sequence of the form 
\begin{equation}
1\to (\BZ^{r})^{k}\stackrel{p_{*}}{\rightarrow}\pi_{1}(\Hom(\BZ^{k},G)_{\BONE})\stackrel{\delta}{\rightarrow} K^{k}\to 1.
\end{equation} 
On the other hand, we also have a covering space $K\stackrel{i}{\rightarrow} \tilde{G}\stackrel{p}{\rightarrow} G$ and the long exact sequence in homotopy associated to this sequence gives a short exact sequence  
\[
1\to \pi_{1}(\tilde{G})\cong \BZ^{r}\stackrel{p_{*}}{\rightarrow}\pi_{1}(G)\stackrel{\delta}{\rightarrow} K\to 1.
\]
By taking the direct sum $k$-copies of this sequence we obtain a short exact sequence 
\begin{equation}
1\to (\BZ^{r})^{k}\stackrel{(p_{*})^{k}}{\rightarrow}(\pi_{1}(G))^{k}\stackrel{(\delta)^{k}}{\rightarrow} K^{k}\to 1.
\end{equation}
We claim that we can find a natural homomorphism $h_{G}$ making the following diagram commuting:

\begin{equation}\label{commutativediagram}
\begin{CD}
1@>>>  (\BZ^{r})^{k} @>{ (p_{*})^{k}}>>
(\pi_{1}(G))^{k}@>(\delta)^{k}>>K^{k}@>>>1\\
@.@V{id}VV @V{h_{G}}VV @V{id}VV\\
1@>>> (\BZ^{r})^{k} @>p_{*}>> \pi_{1}(\Hom(\BZ^{k},G)_{\BONE})@>\delta>> K^{k} @>>>1.
\end{CD}
\end{equation}
Then by the five lemma it follows that 
\[
h_{G}:(\pi_{1}(G))^{k}\to \pi_{1}(\Hom(\BZ^{k},G)_{\BONE})
\]
is an isomorphism hence proving Theorem \ref{main}. 

To construct the homomorphism $h_{G}$, define for every $1\le a\le k$
\begin{align*}
j_{a}:\pi_{1}(G)&\to (\pi_{1}(G))^{k}\\
[\alpha]&\mapsto (1,\dots,[\alpha],\dots,1).
\end{align*}
In other words, $j_{a}$ is the inclusion of $\pi_{1}(G)$ into the $a$-th factor of $(\pi_{1}(G))^{k}$. Notice that the elements in the image of 
$j_{1},\dots,j_{k}$ generate $(\pi_{1}(G))^{k}$, and thus it suffices to define $h$ on elements of the form $j_{a}([\alpha])$ for some $1\le a\le k$ and some loop $\alpha:[0,1]\to G$ based at $1_{G}$. For such elements define 
\[
h(j_{a}([\alpha]))=[i_{a}(\alpha)]\in \pi_{1}(\Hom(\BZ^{k},G)),
\]
where as before \begin{align*}
i_{a}:G&\to \Hom(\BZ^{k},G)\\
g&\mapsto (1_{G},\dots,g,\dots,1_{G})
\end{align*}
the inclusion of $G$ into the $a$-th factor of $\Hom(\BZ^{k},G)\subset G^{k}$. In this way we obtain a well-defined homomorphism 
\[
h_{G}:(\pi_{1}(G))^{k}\to \pi_{1}(\Hom(\BZ^{k},G)).
\]
From the definition it follows at once that $h_{G}$ is a natural map. To see that diagram (\ref{commutativediagram}) commutes note that for every $1\le a\le k$ we have a morphism of fibrations sequences 
\begin{equation*}
\begin{CD}
K @>{i}>>\tilde{G}@>{p}>>G\\
@V{i_{a}}VV @V{i_{a}}VV @V{i_{a}}VV\\
K^{k}@>{i_{*}}>> \Hom(\BZ^{k},\tilde{G})_{\BONE}@>p_{*}>> \Hom(\BZ^{k},G)_{\BONE}.
\end{CD}
\end{equation*}
The naturality of the long exact sequence in homotopy shows that the corresponding diagram in homotopy groups commutes. This diagram is precisely the restriction of
(\ref{commutativediagram}) onto the $a$-th factor. This proves the commutativity of (\ref{commutativediagram}).

\section{Examples and general remarks}\label{examples}

In this section we explore the situation in which the base point of $\Hom(\BZ^{k},G)$ is no longer assumed to be in the path-connected component $\Hom(\BZ^{k},G)_{\BONE}$. For instance, our second example below shows that even if $G$ is simply connected, $\Hom(\BZ^k,G)$ may have connected components with non-trivial $\pi_1$.

\medskip

To start, let $H$ be a compact connected Lie group. As pointed out above, the space $\Hom(\BZ^{k},H)$ is not necessarily connected. This can be explained as follows. Suppose first that $H$ is not simply connected.  Then $H$ can be written in the form $H=G/K$, where $G$ is the universal cover of $H$ and $K\subset G$ is a closed central subgroup. Let 
\[
p:G\to G/K=H
\] 
be the natural projection. Given a commuting sequence $(x_{1},\dots,x_{k})$ in $H$ we can find a lifting $\tilde{x}_{i}$ of $x_{i}$ in $G$ for all $1\le i\le k$. The sequence $(\tilde{x}_{1},\dots,\tilde{x}_{k})\in G^{k}$ is not necessarily a commuting sequence. Instead, $[\tilde{x}_{i},\tilde{x}_{j}]\in K=\Ker(p)\subset Z(G)$. We call such a sequence  a $K$-almost commuting sequence in $G$. Following \cite{ACG}, given a Lie group $G$ and a closed subgroup $K\subset Z(G)$, we denote by $B_{k}(G,K)$ the set of $K$-almost commuting $k$-tuples; that is, the set of sequences $(x_{1},\dots,x_{k})$ such that $[x_{i},x_{j}]\in K$ for all $1\le i,j\le k$. The set $B_{k}(G,K)$ is given the subspace topology under the natural inclusion $B_{k}(G,K)\subset G^{k}$. It is easy to see that projection map  $p:G\to G/K$ induces a $K^{k}$-principal 
bundle
\[
p_{*}:B_{k}(G,K)\to \Hom(\BZ^{k},G/K).
\]
This shows that we can understand the space of commuting elements in $G/K$ by studying the space of $K$-almost commuting elements in $G$. For example, by keeping track of the different commutators of sequences in $B_{k}(G,K)$ this space can be broken down into a disjoint union of subspaces that are both open and closed in $B_{k}(G,K)$, hence a union of path-connected components. Moreover, the image of these components under the map $p_{*}$ provides different path-connected components of the space $\Hom(\BZ^{k},G/K)$. 

\medskip

\noindent{\bf{Example 1:}} Given an integer $m\ge 1$ and any prime number $p$, consider $\SU(p)^{m}$, the product of $m$-copies of $\SU(p)$. Let $\Delta(\BZ/p)$ be the diagonal inclusion of $\BZ/p$ into the center of $\SU(p)^{m}$. Define 
\[
G_{m,p}:=\SU(p)^{m}/\Delta(\BZ/p).
\]
Thus $G_{m,p}$ is the $m$-fold central product of $\SU(p)$. The space of commuting elements in $G_{m,p}$ can be understood by studying the space of almost commuting elements in $\SU(p)^{m}$. Indeed, let  $E_{p}\subset \SU(p)$ be the quaternion group $Q_{8}$ of order eight when $p=2$ and the extraspecial $p$--group of order $p^3$ and exponent $p$ when $p>2$. In \cite{ACG} it was proved that for any $k\ge 1$ the space $\Hom(\BZ^{k},G_{m,p})$ has
\[
N(k,m,p)=\frac{p^{(m-1)(k-2)}(p^{k}-1)(p^{k-1}-1)}{p^{2}-1}+1
\]
path--connected components. One of these path-connected components is $\Hom(\BZ^{k},G_{m,p})_{\BONE}$ and all others are homeomorphic to 
\[
A_{m,p}:=\SU(p)^{m}/((\BZ/p)^{m-1}\times E_{p}).
\]
The path-connected components of $\Hom(\BZ^{k},G_{m,p})$ that are homeomorphic to $A_{m,p}$ have the additional property that the centralizer in $G_{m,p}$ of any sequence in them is a finite group. Let $\x\in \Hom(\BZ^{k},G_{m,p})$ be a point which is taken as the base point of $\Hom(\BZ^{k},G_{m,p})$. Using Theorem \ref{main} it follows that 
\[
\pi_{1}(\Hom(\BZ^{k},G_{m,p}),\x)\cong \pi_{1}(G_{m,p})^{k}\cong (\BZ/p)^{k}
\]
whenever $\x$ lies in $\Hom(\BZ^{k},G_{m,k})_{\BONE}$. On the other hand, if $\x$ lies in a path-connected component of $\Hom(\BZ^{k},G_{m,p})$ that is 
homeomorphic to $A_{m,p}$, then since $SU(p)^{m}$ is simply connected we have that
\[
\pi_{1}(\Hom(\BZ^{k},G_{m,p}),\x)\cong (\BZ/p)^{m-1}\times E_{p}.
\]
Note in particular that $\pi_{1}(\Hom(\BZ^{k},G_{m,p}),\x)$ is independent of $k$ in this case. 

\medskip

Now let's turn our attention to the case of a Lie group $G$ that is assumed to be simply connected. Even in this situation the space $\Hom(\BZ^{k},G)$ is not necessarily path-connected. If fact by \cite[Theorem 4.1]{Baird}, if $G$ is a compact simple Lie group such that $\Hom(\BZ^{k},G)$ is path-connected for every $k\ge 1$, then $G$ is either $\SU(m)$ or $\Sp(m)$ for some $m\ge 1$. Thus in general, the space $\Hom(\BZ^{k},G)$ has many path-connected components. This can be seen as follows. Let $\x:=(x_{1},\dots,x_{k})\in\Hom(\BZ^{k},G)$ and consider the centralizer $Z_{G}(\x)$ of the $k$-tuple $(x_{1},\dots,x_{k})$ in $G$. Let $S$ be a maximal torus in $Z_{G}(\x)$. Proposition \ref{blabla} shows that $\x$ lies in $\Hom(\BZ^{k},G)_{\BONE}$ if and only if $S$ is a maximal torus in $G$. Therefore the space $\Hom(\BZ^{k},G)$ is not path-connected precisely when we can find a commuting $k$-tuple $\x$ such that $Z_{G}(\x)$ does not contain a maximal torus in $G$. The following example, first studied by Kac and Smilga in \cite{KS}, illustrates this possibility.

\medskip

\noindent{\bf{Example 2:}} The space $\Hom(\BZ^{3},\Spin(7))$ has two path-connected components. One of these components is $\Hom(\BZ^{3},\Spin(7))_{\BONE}$ the other component we denote by $B_{3}$.  In \cite{KS}, it was proved directly that in $\Spin(7)$ there is a commuting triple $(x_{1},x_{2},x_{3})$, unique up to conjugation, such that any maximal torus in $Z_{\Spin(7)}(x_{1},x_{2},x_{3})$ has rank $0$, thus explaining the existence of $B_{3}$. This can also be seen in the following way.  As explained in \cite{KS}, we can find an element $x_{1}$ in $\Spin(7)$ such that 
\[
Z_{\Spin(7)}(x_{1})=(\SU(2))^{3}/\Delta(\BZ/2)=G_{3,2}.
\] 
By the previous example, the space $\Hom(\BZ^{2},G_{3,2})$ has two different path-connected components. In particular, we can choose $(x_{2},x_{3})\in \Hom(\BZ^{2},G_{3,2})$ outside the path-connected component containing the trivial representation $\BONE$. As pointed out above, elements in this component have the additional property that $Z_{G_{3,2}}(x_{2},x_{3})$ is a finite group. This shows that any maximal torus in $Z_{\Spin(7)}(x_{1},x_{2},x_{3})$ has rank $0$, as any maximal torus in $Z_{G_{3,2}}(x_{2},x_{3})$ already has rank $0$, hence explaining the existence of an {\em exotic} path-connected component in $\Hom(\BZ^{3},\Spin(7))$. Moreover, the triple $(x_{1},x_{2},x_{3})$ is unique up to conjugation in $\Spin(7)$. This shows that the conjugation action of $\Spin(7)$ on $B_{3}$ is transitive; in particular there is a homeomorphism 
\[
B_{3}\cong \Spin(7)/Z_{\Spin(7)}(x_{1},x_{2},x_{3}).
\]
Using the work in \cite{ACG} it is easy to see that 
\[
Z_{\Spin(7)}(x_{1},x_{2},x_{3})=Z_{G_{3,2}}(x_{2},x_{3})\cong (\BZ/2)^{4}.
\] 
This shows that 
\begin{equation}\label{exotic}
B_{3}\cong \Spin(7)/(\BZ/2)^{4}\cong \SO(7)/(\BZ/2)^{3},
\end{equation}
for some embedding $(\BZ/2)^{3}\hookrightarrow \SO(7)$. Using this and Theorem \ref{main}, we see that if $\y\in \Hom(\BZ^{3},\Spin(7))$ is taken as the base point, then 
\[
\pi_{1}(\Hom(\BZ^{3},\Spin(7)),\y)=1
\]
whenever $\y\in \Hom(\BZ^{3},\Spin(7))_{\BONE}$. In contrast, if ${\bf y} \in B_{3}$ then by (\ref{exotic}) 
\[
\pi_{1}(\Hom(\BZ^{3},\Spin(7)),\y)=(\BZ/2)^{4}.
\]

\bigskip

Examples 1 and 2 show that Theorem \ref{main} may not hold if the base point of $\Hom(\BZ^{k},G)$ is no longer assumed to be in $\Hom(\BZ^{k},G)_{\BONE}$.

\bigskip

\noindent{\small Department of Mathematics, University of British Columbia, Vancouver \newline \noindent
\texttt{josmago@math.ubc.ca}}
\medskip

\noindent{\small Department of Mathematics, University of Michigan, Ann Arbor \newline \noindent
\texttt{apettet@umich.edu}}
\medskip

\noindent{\small Department of Mathematics, University of Michigan, Ann Arbor \newline \noindent
\texttt{jsouto@umich.edu}}

\end{document}